\documentclass[12pt, reqno]{amsart}
\usepackage{amsmath, amsthm, amscd, amsfonts, amssymb, graphicx, color}
\usepackage[bookmarksnumbered, colorlinks, plainpages]{hyperref}
\hypersetup{colorlinks=true,linkcolor=red, anchorcolor=green, citecolor=cyan, urlcolor=red, filecolor=magenta, pdftoolbar=true}

\usepackage{amsfonts}
\usepackage{color,graphicx,shortvrb}
\usepackage{enumerate}
\usepackage{amsmath}
\usepackage{amssymb}

\newtheorem{thm}{Theorem}
\newtheorem{corollary}[thm]{Corollary}
\newtheorem{lemma}[thm]{Lemma}
\newtheorem{prop}[thm]{Proposition}
\theoremstyle{definition}

\theoremstyle{remark}

\numberwithin{equation}{section}

\def\J#1#2#3{ \left\{ #1,#2,#3 \right\} }

\usepackage[latin 1]{inputenc}

\begin{document}
\title[\v{C}eby\v{s}\"{e}v subtriples of JB$^{*}$-triples]{Inner ideals, compact tripotents and \v{C}eby\v{s}\"{e}v subtriples of JB$^{*}$-triples and C$^*$-algebras}
\author[F.B. Jamjoom]{Fatmah B. Jamjoom}
\address{Department of Mathematics, College of Science, King Saud
University, P.O.Box 2455-5, Riyadh-11451, Kingdom of Saudi Arabia.}
\email{fjamjoom@ksu.edu.sa}
\author[A.M. Peralta]{Antonio M. Peralta}
\address{Departamento de An{\'a}lisis Matem{\'a}tico, Universidad de Granada,%
\\
Facultad de Ciencias 18071, Granada, Spain}
\email{aperalta@ugr.es}
\author[A.A. Siddiqui]{Akhlaq A. Siddiqui}
\address{Department of Mathematics, College of Science, King Saud
University, P.O.Box 2455-5, Riyadh-11451, Kingdom of Saudi Arabia.}
\email{asiddiqui@ksu.edu.sa}
\author[H.M. Tahlawi]{Haifa M. Tahlawi}
\address{Department of Mathematics, College of Science, King Saud
University, P.O.Box 2455-5, Riyadh-11451, Kingdom of Saudi Arabia.}
\email{htahlawi@ksu.edu.sa}

\begin{abstract} The aim of this note is to study \v{C}eby\v{s}\"{e}v JB$^*$-subtriples of general JB$^*$-triples. It is established that if $F$ is a non-zero  \v{C}eby\v{s}\"{e}v JB$^*$-subtriple of a JB$^*$-triple $E$, then exactly one of the following statements holds:\begin{enumerate}[$(a)$]\item $F$ is a rank one JBW$^*$-triple with dim$(F)\geq 2$ (i.e. a complex Hilbert space regarded as a type 1 Cartan factor). Moreover, $F$ may be a closed subspace of arbitrary dimension and $E$ may have arbitrary rank;
\item $F= \mathbb{C} e$, where $e$ is a complete tripotent in $E$;
\item $E$ and $F$ are rank two JBW$^*$-triples, but $F$ may have arbitrary dimension;
\item $F$ has rank greater or equal than three and $E=F$.
\end{enumerate}
\end{abstract}

\thanks{The authors extend their appreciation to the Deanship of Scientific Research at King Saud University for funding this work through research group no  RG-1435-020. The second author also is partially supported by the Spanish Ministry of Economy and Competitiveness project no. MTM2014-58984-P}

\maketitle
\section{Introduction}

It is known that certain problems on operator algebras are more feasible when the algebra under study is a von Neumann algebra (i.e. a C$^*$-algebra which is also a dual Banach space). For example, A.G. Robertson gave in \cite{Robert1977} a complete description of one-dimensional \v{C}eby\v{s}\"{e}v subspaces, and the finite dimensional \v{C}eby\v{s}\"{e}v hermitian subalgebras with dimension bigger than 1 of a general von Neumann algebra. Concretely, for a non-zero element $x$ in a von Neumann algebra $M$, subspace $\mathbb{C} x$ is a \v{C}eby\v{s}\"{e}v subspace of ${M}$ if and only if there is a projection $p$ in the
center of ${M}$ such that $p x$ is left invertible in $p{M}$ and $(1-p)x$ is right invertible in $(1-p){M}$ (cf. \cite[Theorem 1]{Robert1977}). A finite dimensional $^*$-subalgebra $N$ of an infinite dimensional von Neumann algebra ${M}$ with dim$(N)>1$ never is a \v{C}eby\v{s}\"{e}v subspace of ${M}$ (see \cite[Theorem 6]{Robert1977}).\smallskip

Two years were needed to relax the assumptions concerning duality, to finally obtain valid answers for \v{C}eby\v{s}\"{e}v subspaces and subalgebras of a general C$^*$-algebra. A.G. Robertson and D. Yost proved in \cite[Corollary 1.4]{RobertYost79} that an infinite dimensional C$^*$-algebra $A$ admits a finite dimensional $^*$-subalgebra $B$ which is also a \v{C}eby\v{s}\"{e}v in $A$ if and only if $A$ is unital and $B=\mathbb{C} 1$. The results proved by Robertson and Yost were complemented by G.K. Pedersen, who shows that if $A$ is a C$^*$-algebra without unit and $B$ is a \v{C}eby\v{s}\"{e}v C$^*$-subalgebra of $A$, then $A=B$ (compare \cite[Theorem 4] {Ped79}).\smallskip

We recall that a subspace $V$ of a Banach space $X$ is called a \emph{\v{C}eby\v{s}\"{e}v {\rm(}Chebyshev{\rm)} subspace} of $X$ if for each $x\in X$ there exists a unique point $\pi_{_V} (x)\in V$ such that $\hbox{dist}(x,V)=\left\Vert x-\pi_{_V} (x)\right\Vert $. Throughout this note the symbol $\pi_{_V} (x)$ will denote the best approximation of an element $x$ in $X$ in a \v{C}eby\v{s}\"{e}v subspace $V$ of $X$. For more information on \v{C}eby\v{s}\"{e}v and best approximation theory we refer to the monograph \cite{Singer}.\smallskip

Similar benefits to those obtained working with von Neumann algebras re-appear when studying \v{C}eby\v{s}\"{e}v subspaces which Ternary Rings of Operators (TRO's) of a given von Neumann algebra, or when exploring \v{C}eby\v{s}\"{e}v JBW$^*$-subtriples of a given JBW$^*$-triple (see section \ref{sec:preliminaries} for definitions). In a previous paper, we establish the following description of \v{C}eby\v{s}\"{e}v JBW$^*$-subtriples of a JBW$^*$-triple.

\begin{thm}\label{t Cebysev  JBW*-subtriples JPST}\cite[Theorem 13]{JamPeSiTah2014Ceby} Let $N$ be a non-zero \v{C}eby\v{s}\"{e}v JBW$^*$-subtriple of a JBW$^*$-triple $M$. Then exactly one of the following statements holds:\begin{enumerate}[$(a)$]\item $N$ is a rank one JBW$^*$-triple with dim$(N)\geq 2$ (i.e. a complex Hilbert space regarded as a type 1 Cartan factor). Moreover, $N$ may be a closed subspace of arbitrary dimension and $M$ may have arbitrary rank;
\item $N= \mathbb{C} e$, where $e$ is a complete tripotent in $M$;
\item $N$ and $M$ have rank two, but $N$ may have arbitrary dimension;
\item $N$ has rank greater or equal than three and $N=M$. $\hfill\Box$
\end{enumerate}
\end{thm}

The question whether in the above theorem JBW$^*$-triples and subtriples can be replaced with JB$^*$-triples and subtriples remains as an open problem. The techniques employed in \cite{JamPeSiTah2014Ceby} rely heavily on the rich geometric properties of JBW$^*$-triples. In this note we study this problem in the more general setting of JB$^*$-triples. We combine here new arguments involving inner ideals and compact tripotents in the bidual of a JB$^*$-triple. The main result of this note shows that the conclusion of the above Theorem \ref{t Cebysev  JBW*-subtriples JPST} also holds when $N$ is a JB$^*$-subtriple of a general JB$^*$-triple $M$ (see Theorem \ref{t Cebysev JB*-subtriples JPST}).\smallskip

Among the new results proved in this note we also establish that a \v{C}eby\v{s}\"{e}v C$^*$-subalgebra $B$ (respectively, a \v{C}eby\v{s}\"{e}v JB$^*$-subtriple) of a C$^*$-algebra $A$ with rank$(B)\geq 3$ coincides with the whole $A$ (see Corollary \ref{c t F rank 3 Cebysev subtriple}).

\section{Preliminaries}\label{sec:preliminaries}

The multiple attempts to understand a Riemann mapping theorem type for complex Banach spaces of dimension bigger or equal than 2, led some mathematicians to the study of bounded symmetric domains (compare \cite{Cartan35, Loos77, Harr74, Harr81} and \cite{Ka83}). The definite answer was given by W. Kaup, who showed the existence of a set of algebraic-geometric-analytic axioms which determine a class of complex Banach spaces, the class of JB$^*$-triples, whose open unit balls are bounded symmetric domains, and every bounded symmetric domain in a complex Banach space is biholomorphically equivalent to the open unit ball of a
JB$^*$-triple; in this way, the category of all bounded symmetric domains with base point is equivalent to the category of JB$^*$-triples.\smallskip

A \emph{JB$^{*}$-triple} is a complex a Banach space $E$ with a continuous triple product $(a,b,c) \mapsto \{a,b,c\}$, which is bilinear and symmetric in the external variables and conjugate linear in the middle one and satisfies:
\begin{enumerate}[$(a)$]\item (Jordan identity)
$$L(x,y)\{a,b,c\}=\{L(x,y)a,b,c\}-\{a,L(y,x)b,c\}+\{a,b,L(x,y)c\},$$ for all $x,y,a,b,c\in E$, where $L(x,y):E\rightarrow E$ is the linear
mapping given by $L(x,y)z=\{x,y,z\}$;
\item  For each $x\in E$, the operator $L(x,x)$ is hermitian with non-negative spectrum;
\item $\left\Vert \{x,x,x\}\right\Vert =\left\Vert x\right\Vert ^{3}$ for all $
x\in E$.
\end{enumerate}

Given an element $a$ in a JB$^*$-triple $E$, the symbol $Q(a)$ will denote the conjugate linear map on $E$ defined by $Q(a) (x) := \{a,x,a\}$.\smallskip

The class of JB$^*$-triples includes all C$^*$-algebras when the latters are equipped with the triple product given by \begin{equation}\label{eq ternary product on C*-algebras} \{a,b,c\} = \frac12 (a b^* c+ c b^* a).
\end{equation} The space $B(H,K)$ of all bounded linear operators between complex Hilbert spaces, although rarely is a C$^*$-algebra, is a JB$^*$-triple with the product defined in \eqref{eq ternary product on C*-algebras}. In particular, every complex Hilbert space is a JB$^*$-triple. Thus, the class of JB$^*$-triples is strictly wider than the class of C$^*$-algebras.\smallskip

A JBW$^*$-triple is a JB$^*$-triple which is also a dual Banach space (with a unique isometric predual \cite{Barton}). The triple product of every JBW$^*$-triple
is separately weak$^*$ continuous (cf. \cite{Barton}). The second dual, $E^{**},$ of a JB$^*$-triple, $E,$ is a JBW$^*$-triple with a certain triple product extending the product of $E$ (cf. \cite{Di86b}).\smallskip

The study of \v{C}eby\v{s}\"{e}v JB$^*$-subtriples of a given JB$^*$-triple requires some other examples which are very well known. A Cartan factor of type 1 is a JB$^*$-triple which coincides with the Banach space $B(H, K)$ of bounded linear operators between two complex Hilbert spaces, $H$ and $K$, where the triple product is defined by \eqref{eq ternary product on C*-algebras}. Cartan factors of types 2 and 3 are JB$^*$-triples which can be identified the subtriples of $B(H)$ defined by $II^{\mathbb{C}} = \{ x\in B(H) : x=- j x^* j\} $ and $III^{\mathbb{C}} = \{ x\in B(H) : x= j x^* j\}$, respectively, where $j$ is a conjugation on $H$. A Cartan factor of type 4 or $IV$ is a spin factor, that is, a complex Hilbert space provided with
a conjugation $x \mapsto \overline{x}$, where the triple product and the norm are defined by $$\J x y z
= \langle x / y \rangle z + \langle z / y \rangle x - \langle x /
\bar z \rangle \bar y,$$ and $\| x\|^2=\langle x / x
\rangle+\sqrt {\langle x / x \rangle^2-|\langle x / \overline x
\rangle|^2}$, respectively. The Cartan factors of types 5 and 6 consist of finite dimensional spaces of matrices over the eight dimensional complex Cayley division algebra $\mathbb{O}$; the type $VI$ is the space of all hermitian $3$x$3$ matrices over $\mathbb{O}$, while the type $V$ is the subtriple of $1$x$2$ matrices with entries in $\mathbb{O}$ (compare \cite{Loos77}, \cite{FriRu86}, and \cite[\S 2.5]{Chu}).\smallskip

Let $E$ be a JB$^*$-triple. An element $e\in E$ is called a \emph{tripotent} if $\J eee =e$. For each tripotent $e\in E$, the eigenspaces of the operator $L(e,e)$ induce a decomposition (called \emph{Peirce decomposition}) of $E$ in the form
$$E= E_{2} (e) \oplus E_{1} (e) \oplus E_0 (e),$$ where for $i=0,1,2,$ $E_i (e)=\{ x\in E : L(e,e) (x) = \frac{i}{2} x\}$ (compare \cite[Theorem 25]{Loos77}). The natural projections of $E$ onto $E_i(e)$ will be denoted by $P_i(e)$.  It is known that this decomposition satisfies the following multiplication rules:
 $$\J {E_{i}(e)}{E_{j}
(e)}{E_{k} (e)}\subseteq E_{i-j+k} (e),$$ if $i-j+k \in \{
0,1,2\}$ and is zero otherwise. In addition, $$\J {E_{2}
(e)}{E_{0}(e)}{E} = \J {E_{0} (e)}{E_{2}(e)}{E} =0.$$

A tripotent $e$ in $E$ is called \emph{complete} (respectively, \emph{minimal}) if the equality $E_0(e)=0$ (respectively, $E_2(e)=\mathbb{C} e \neq \{0\}$) holds.\smallskip

The connections between JB$^*$-triples and JB$^*$-algebras are very deep. Every JB$^*$-algebra is a JB$^*$-triple under the triple product defined \begin{equation}\label{eq JB*-alg-triple product} \J xyz = (x\circ y^*) \circ
z + (z\circ y^*)\circ x - (x\circ z)\circ y^*.
\end{equation} The Peirce space $E_2 (e)$ is a JB$^*$-algebra with product
$x\circ_e y := \J xey$ and involution $x^{\sharp_e} := \J exe$.
\smallskip

Let $a$ be an element in a JB$^*$-triple $E$. It is known that the JB$^*$-subtriple, $E_{a}$, generated by $a$, identifies with some $C_0(L_a)$ where $\|a\|\in L_a\subseteq [0,\|a\|]$ with $L_a\cup \{0\}$ compact (cf. \cite[1.15]{Ka83}). Moreover, there exists a triple isomorphism $\Psi: E_{a} \to C_0(L_a)$ such that $\Psi (a) (t) =t$. Consequently, the symbol $a^{\frac{1}{3}}$ will stand for the unique element in $E_a$ satisfying $\{a^{\frac{1}{3}},a^{\frac{1}{3}},a^{\frac{1}{3}}\} = a$.\smallskip

When $a$ is an element in a JBW$^{*}$-triple $M$, the sequence $(a^{\frac{1}{2n-1}})$ converges in the weak$^{\ast}$-topology
of $M$ to a tripotent, denoted by $r(a)$, called the \textit{range tripotent of} $a$. The tripotent $r(a)$ is the smallest tripotent $e\in M$ satisfying
that $a$ is positive in the JBW$^{\ast }$-algebra $M_{2}(e)$ (see \cite[page 322]{Edwards}). Clearly, the range tripotent $r(a)$ can be identified with the characteristic function $\chi_{_{(0,\|a\|]\cap L_a}}\in C_0(L_a)^{**}$ (see \cite[beginning of \S 2]{BuChuZa2000}).\label{page triple functional calculus} If we define $a^{[1]} =a$, $a^{[3]} = \{a,a,a\}$ and $a^{[2n+1]} = \{a,a,a^{[2n-1]}\}$, for every $n\geq 1$, it is known that the sequence $(a^{[2n -1]})$ converges in the weak$^*$ topology of $M$ to a tripotent (called the \hyphenation{support}\emph{support} \emph{tripotent} of $a$) $s(a)$ in $E^{**}$, which satisfies $ s(a) \leq a \leq r(a)$ in the JBW$^*$-algebra $M_2 (r(a))$ (compare \cite[Lemma 3.3]{Edwards}).\smallskip

We recall that two elements $a, b$ in a JB$^*$-triple $E$ are \emph{orthogonal} (written as $a\perp b$) if $L(a, b) = 0$ (see \cite[Lemma 1]{BurFerGarMarPe08} for several equivalent reformulations). Given a subset $M\subseteq E$, we write $M_{E}^{\perp }$
(or simply $M^{\perp }$) for the (orthogonal) annihilator of $M$ defined by
$M_{E}^{\perp }=\{y\in E:y\perp x,\forall x\in M\}$. If $e\in E$ is a tripotent,
then $\{e\}_{E}^{\perp }=$ $E_{0 }(e)$, and $\{a\}_{E}^{\perp }=$ $(E^{\ast \ast })_{0 }(r(a))\cap
E$, for every $a\in E$ (cf. \cite[Lemma 3.2]{Burgos2}).\smallskip

Given a tripotent $e\in E$, we know from Lemma 1.3$(a)$ in \cite{FriRu85} that $$\|x_2 + x_0\| = \max \{ \|x_2\|, \|x_0\| \},$$ for every $x_2\in E_2 (e) $ and every $x_0\in E_0 (e).$ This geometric property with the equivalent reformulations of orthogonality given in \cite[Lemma 1]{BurFerGarMarPe08} we deduce that \begin{equation}\label{eq orthogonal are M-orthogonal} \|a + b\|  = \max \{ \|a\|, \|b\| \},
\end{equation} whenever $a$ and $b$ are orthogonal elements in a JB$^*$-triple. A subset $S \subseteq E$ is said to be \emph{orthogonal} if $0 \notin S$ and $x \perp y$ for every $x \neq y$ in $S$. The minimal cardinal number $r$ satisfying $card(S) \leq r$ for every orthogonal subset $S \subseteq E$ is called the \emph{rank} of $E$ (and will be denoted by $r(E)$). Given a tripotent $e\in E$,  the rank of the Peirce-2 subspace $E_2(e)$ will be called the rank of $e$.\smallskip

We shall also make use of a natural partial order defined on the set of tripotents (see Corollary 1.7 and comments preceding it in \cite{FriRu85}). Given two tripotents $e,u$ in a JB$^*$-triple $E$, we say that $e\leq u$ if $u-e$ is a tripotent in $E$ with $u-e\perp e$.\smallskip

We finally, recall that an element $x$ in a JB$^*$-triple $E$ is called \emph{Brown-Pedersen quasi-invertible} (BP quasi-invertible for short) if there exists $y\in E$ such that $B(x,y)=0$ (cf. \cite{JamSiddTah2013}), where $B(x,y)$ denotes the Bergmann operator $B(x,y)=I_E-2L(x,y)+Q(x)Q(y)$. Theorems 6 and 11 in \cite{JamSiddTah2013} prove that an element $x$ in $E$ is Brown-Pedersen quasi-invertible if, and only if, $x$ is von Neumann regular in the sense of \cite{FerGarSanMo92,Ka96,BurKaMoPeRa} and its range tripotent is an extreme point of the closed unit ball of $E$, equivalently, there exists a complete tripotent $v\in E$ such that $x$ is positive and invertible in $E_2(v)$. In particular, every extreme point of the closed unit ball of $E$ is BP quasi-invertible.\smallskip

\section{\v{C}eby\v{s}\"{e}v subtriples of JB$^*$-triples}\label{sec: Cebysev JB*-subtriples}

The following auxiliary results were established in \cite[\S 3]{JamPeSiTah2014Ceby}

\begin{prop}\label{prop JamPeSiddTahCeby}\cite[Propositions 9 and 10]{JamPeSiTah2014Ceby} Let $F$ be a \v{C}eby\v{s}\"{e}v  JB$^*$-subtriple of a JB$^*$-triple $E$. Suppose $e$ is a non-zero tripotent in $F$.
Then the following statements hold: \begin{enumerate}[$(a)$]\item $E_0 (e) = F_0 (e)$, and consequently, every complete tripotent in $F$ is complete in $E$.
\item If $F_0(e) = \{e\}^{\perp}_{F}\neq 0$, then $E_2 (e) = F_2 (e)$.$\hfill\Box$
\end{enumerate}
\end{prop}

We continue, in this paper, our study on \v{C}eby\v{s}\"{e}v subtriples of general JB$^*$-triples.\smallskip

We recall that orthogonal elements in $E$ are always $M$-orthogonal, i.e. $a\perp b$ in $E$ implies that $\|\alpha a + \beta b\| = \max\{\|\alpha a \|, \| \beta b\| \},$ for all $\alpha,\beta\in \mathbb{C}$ (see, for example, \cite[Lemma 1]{BurFerGarMarPe08} and \cite[Lemma 1.3]{FriRu85}).

\begin{prop}\label{prop inner ideals} Let $F$ be a \v{C}eby\v{s}\"{e}v  JB$^*$-subtriple of a JB$^*$-triple $E$. For each $a$ in $F$ which is not Brown-Pedersen quasi-invertible the norm closure of the space $Q(a) (E)$ is contained in $F$.
\end{prop}

\begin{proof} Since $a$ is in $F\backslash F_q^{-1}$ we have two possibilities either $a$ is not von Neumann regular or $a$ is von Neumann regular and its range tripotent is not an extreme point of the closed unit ball of $F$. We deal with each case separately. We can assume that $\|a\|=1$.\smallskip

Suppose first that $a$ is not von Neumann regular. Then $0$ is a non-isolated point in the triple spectrum $L_{a}$ of $a$. We know that in this case, $0,1\in L_a\subseteq [0,1]$, with $L_a$ compact. We further know that, if $F_a$ denotes the JB$^*$-subtriple of $F$ generated by $a$, then there exists a triple isomorphism $\varPsi$ from $F_a$ onto $C_{0}(L_a)$, satisfying $\varPsi(a) (t) = t$ $(t\in L_a),$ where $C_0(L_a)$ denotes the commutative C$^*$-algebra of all continuous functions on $L_a$ vanishing at zero (cf. \cite[Lemma~1.14]{Ka83} and \cite[\S 3]{Ka96}).\smallskip

Fix a natural $n$, and define the following functions in $F_a$, defined as elements in $C_0(L_a)$, $$b_{n} (t):=\left\{%
\begin{array}{ll}
    0, & \hbox{if $t\in [0,\frac1{3n} ]$;} \\
    \hbox{affine}, & \hbox{if $t\in [\frac1{3n},\frac1{2n} ]$;} \\
    1, & \hbox{if $t\in [\frac1{2n},1 ]$} \\
\end{array}%
\right. a_{n} (t):=\left\{%
\begin{array}{ll}
    0, & \hbox{if $t\in  [0,\frac1{2n} ]$;} \\
    \hbox{affine}, & \hbox{if $t\in [\frac1{2n},\frac1{n} ]$;} \\
    t, & \hbox{if $t\in  [\frac1{n},1 ]$} \\
\end{array}%
\right.$$ and $$z_{n} (t):=\left\{%
\begin{array}{ll}
    0, & \hbox{if $t=0$;} \\
    1, & \hbox{if $t=\frac1{5n}$;} \\
    \hbox{affine}, & \hbox{otherwise;} \\
    0, & \hbox{if $t\in [\frac1{4n},1]$.} \\
\end{array}%
\right.$$ Clearly $\|a_n\|=\|z_n\|=\|b_n\|=1$, $\|a-a_n\| =\frac1{2n}$, $\{b_n,a_n,b_n\}=a_n.\,$ and $z_n\perp a_n,b_n$, for every $n\in \mathbb{N}$. We claim that \begin{equation}\label{eq an E in F}
\{a_n,E,a_n\}\subseteq F,
\end{equation} for every natural $n$. Suppose, on the contrary, that there exists an element $w\in \{a_n,E,a_n\}\backslash F$. Since in $F_a^{**}$ (and hence in $F^{**}$) $b_n = r(a_n) + (b_n -r(a_n))$, where $r(a_n) \perp (b_n -r(a_n))$ and $r(a_n)$ is the range tripotent of $a_n$ in $F^{**}$, and $w\in\{a_n,E,a_n\}$, we can easily see that $Q(b_n)^2 (w) =w$. The element $\pi_{_F} (w)$ lies in $F$, and thus $Q(b_n)^2 (\pi_{_F} (w))\in F$. We observe that $$\left\| w - Q(b_n)^2 (\pi_{_F} (w))\right\| = \left\| Q(b_n)^2 ( w - \pi_{_F} (w))\right\|\leq \left\|  w - \pi_{_F} (w) \right\|=\hbox{dist} (w, F).$$ The uniqueness of the best approximation of $w$ in $F$ implies that $\pi_{_F} (w) = Q(b_n)^2 (\pi_{_F} (w))$, equality which implies that $z_n\perp Q(b_n)^2 (\pi_{_F} (w)) =\pi_{_F} (w)$. We therefore have $z_n\perp w+ Q(b_n)^2 (\pi_{_F} (w))$ (because $w=Q(b_n)^2 (w)\perp z_n$). For each $\lambda\in \mathbb{C}$, the element $\lambda z_n + \pi_{_F} (w)$ belongs to $F$, and since orthogonal elements are $M$-orthogonal, we have $$\left\| w- (\lambda z_n + \pi_{_F} (w)) \right\| = \left\| - \lambda z_n + (w- \pi_{_F} (w)) \right\|$$ $$ =\max\{ \|w- \pi_{_F} (w)\|, \| \lambda z_n \|\} =\hbox{dist} (w,F),$$ for every $|\lambda|\leq \|w- \pi_{_F} (w)\| =\hbox{dist} (w,F),$ which contradicts the uniqueness of $\pi_{_F} (w)$. This proves the claim.\smallskip

Now, since $(a_n)\to a$ in norm, the triple product of $E$ is norm continuous, and $F$ is norm closed, we deduce from \eqref{eq an E in F} that ${\{a_n,E,a_n\}}\subseteq F$, and hence $\overline{\{a_n,E,a_n\}}\subseteq F$.\smallskip

Suppose now that $a$ is von Neumann regular but $r(a)\notin \partial_e (F_1),$ that is, $r(a)\in F$ is not a complete tripotent, or equivalently, $F_0(r(a))\neq \{0\}$. By Proposition \ref{prop JamPeSiddTahCeby}$(b)$ (cf. \cite[Proposition 10]{JamPeSiTah2014Ceby}) we have $E_2(r(a)) =F_2(r(a))\subseteq F.$ We note that $a$ also is von Neumann regular in $E$. Finally, it is known that for a von Neumann regular element $a$ in a JB$^*$-triple $E$ we have $Q(a) (E) =Q(r(a)) (E) =E_2 (r(a))$ (cf. \cite[comments after Lemma 3.2]{Kaup01} or \cite[p. 192]{BurKaMoPeRa}). This shows that $Q(a) (E) = E_2 (r(a)) =F_2(r(a))\subseteq F.$
\end{proof}

A (closed) subtriple $I$ of a JB$^*$-triple $E$ is said to be a \emph{triple ideal} or imply an \emph{ideal} of $E$ if $\J EEI + \J EIE\subseteq I$. If we only have $\J IEI \subseteq I$ we say that $I$ is an \emph{inner ideal} of $E$. Following standard notation, given an element $a$ in $E$, we denote by $E(a)$ the norm-closure of $Q(a) (E) = \{a,E,a\}$ in $E$. It is known that $E(a)$ is precisely the norm-closed inner ideal of $E$ generated by
$a$  (cf. \cite{BuChuZa2000}).\label{inner ideal} 
\smallskip

As we commented above, let $E_a$ the JB$^*$-subtriple of $E$ generated by $a$. Clearly, $a^{[3]} = Q(a) (a)\in E_a\subseteq E(a)$, and hence $E(a^{[3]}) \subseteq E(a)$. By the Gelfand theory for JB$^*$-triples, there exist a subset $L_a\subseteq [0,\|a\|]$, with $L_a$ compact and $\|a\|\in L_a$, and a triple isomorphism $\varPsi$ from $E_a$ onto $C_{0}(L_a)$, satisfying $\varPsi(a) (t) = t$ $(t\in L_a)$ (cf. \cite[Lemma~1.14]{Ka83}). By the Stone-Weierstrass theorem, we know that $E_{a} = E_{a^{[3]}}$. Therefore $a\in E_{a^{[3]}}\subseteq E(a^{[3]}),$ which implies that $E(a^{[3]}) \supseteq E(a).$ Therefore $E(a^{[3]}) = E(a).$
\smallskip

The previous Proposition \ref{prop inner ideals} implies that, if $F$ is a \v{C}eby\v{s}\"{e}v JB$^*$-subtriple of a JB$^*$-triple $E$, then for each  non Brown-Pedersen quasi-invertible element $a$ in $F$, the inner ideal of $E$ generated by $a$ is contained in $F$. We can actually prove a stronger conclusion.

\begin{corollary}\label{c prop inner ideals} Let $F$ be a \v{C}eby\v{s}\"{e}v  JB$^*$-subtriple of a JB$^*$-triple $E$. For each $a$ in $F$ which is not Brown-Pedersen quasi-invertible in $F$ the inner ideal generated by $a$ in $E$ coincides with the inner ideal of $F$ generated by $a$, that is, $E(a)=F(a)$.
\end{corollary}

\begin{proof} Let us observe that $a$ is not Brown-Pedersen quasi-invertible in $F$ if and only if $a^{[3]}$ satisfies the same property. Proposition \ref{prop inner ideals} shows that $E(a)$ $ =\overline{Q(a) (E)}\subseteq F.$ That is, $E(a)$ is an inner ideal of $F$ which contains $a,$ and hence $F(a) \subseteq E(a).$\smallskip

Now, we fix $x\in E$. The same Proposition \ref{prop inner ideals} above implies that $Q(a^{[3]}) (x) = Q(a) Q(a) Q(a) (x)\subseteq Q(a) (F)\subseteq F(a),$ and consequently, $E(a)= E(a^{[3]}) = \overline{ Q(a^{[3]}) (E)} \subseteq F(a).$
\end{proof}

Let us observe that in Theorem \ref{t Cebysev  JBW*-subtriples JPST} cases $(a)$, $(b)$ and $(c)$, the JBW$^*$-subtriple $F$ contains BP quasi-invertible elements. The remaining case $(d)$ motivates us to consider the following partial result to determine the \v{C}eby\v{s}\"{e}v JB$^*$-subtriples of a  general JB$^*$-triple.

\begin{corollary}\label{c cebysev subtriples without BP quasi invertible elements} Let $F$ be a \v{C}eby\v{s}\"{e}v  JB$^*$-subtriple of a JB$^*$-triple $E$. Suppose $F$ contains no BP quasi-invertible elements {\rm(}equivalently, $\partial_e (F_1)=\emptyset${\rm)}. Then $F$ is an inner ideal of $E$. \end{corollary}

\begin{proof} By hypothesis $F$ contains not BP quasi-invertible elements, then Proposition \ref{prop inner ideals} implies that $\{a,E,a\} \subseteq F$ for every $a\in F$. A standard polarization argument shows that $\{a,E,b\} \subseteq F$ for every $a,b\in F$, witnessing that $F$ is an inner ideal of $E$.
\end{proof}

\begin{lemma}\label{l distance to extreme points in E} Let $F$ be a JB$^*$-subtriple of a JB$^*$-triple $E$. Suppose $F$ contains no BP quasi-invertible elements {\rm(}equivalently, $\partial_e (F_1)=\emptyset${\rm)}. Then for each $e\in \partial_e(E_1)$ we have $\hbox{\rm dist}(e,F)= 1$. If $F$ is a  \v{C}eby\v{s}\"{e}v subspace of $E$, we have $\pi_{_F} (e) =0$, for every $e$ as above.
\end{lemma}

\begin{proof} Suppose we can find $x\in F$ satisfying $\|e-x\| < 1$. Then $$\left\|e-P_2(e) (x)\right\|= \left\|P_2 (e) (e-x)\right\|\leq \| e-x\|<1.$$ Since $e$ is the unit element of the JB$^*$-algebra $E_2(e)$, we deduce that $P_2 (e) (x)$ is an invertible element in $E_2(e)$. Lemma 2.2 in \cite{JamPeSiTah2014Quart} implies that $x$ is BP quasi-invertible in $E$, and hence BP quasi-invertible in $F$, which is impossible. The second statement is clear because dist$(e,F)=1=\|e\|$.
\end{proof}

The following technical lemma will play an useful role in subsequent results.

\begin{lemma}\label{l surjective linear isometries} Let $V$ be a closed \v{C}eby\v{s}\"{e}v  subspace of a a Banach space $X$. Let $S: X\to X$ be a surjective linear isometry satisfying $S(V), S^{-1} (V) \subseteq V.$ Then $S (\pi_{_V} (x)) = \pi_{_V} (S(x))$, for every $x\in X$.
\end{lemma}

\begin{proof} We claim that $$\hbox{dist} (x,V) = \hbox{dist} (S(x),V),$$ for every $x\in X$. Let us pick $x\in X.$ Since $S$ is a surjective linear isometry, $\hbox{dist} (x,V) = \|x-\pi_{_V} (x) \| =  \|S(x) -S(\pi_{_V} (x)) \|,$ with $S(\pi_{_V} (x))\in V.$ Therefore $$\hbox{dist} (x,V)\geq \hbox{dist} (S(x),V).$$ Applying the same arguments to $S^{-1}$, we deduce that $$\hbox{dist} (x,V)\geq \hbox{dist} (S(x),V) \geq \hbox{dist} (S^{-1} S(x),V) = \hbox{dist} (x,V).$$
\end{proof}

We recall that, given a tripotent $e$ in a JB$^*$-triple $E$ and $\lambda\in \mathbb{C}$ with $|\lambda|=1$, the mapping $$S_{\lambda}: E\to E, S_{\lambda} = \lambda^2 P_2 (e) + \lambda P_1 (e) + P_0 (e)$$ is an isometric triple isomorphism (compare \cite[Lemma 1.1]{FriRu85}). \smallskip

We establish now an strengthened version of Proposition \ref{prop JamPeSiddTahCeby}.

\begin{prop}\label{p orthogonal} Let $F$ be a \v{C}eby\v{s}\"{e}v  JB$^*$-subtriple of a JB$^*$-triple $E$. Suppose that $a$ is a non-zero element in $F$. Then $\{a\}_{E}^{\perp} =\{x\in E : x\perp a\}\subseteq F$.
\end{prop}

\begin{proof} Arguing by contradiction, we suppose the existence of an element $x\in \{a\}_{E}^{\perp} \backslash F.$ From the axioms in the definition of JB$^*$-triples, we know that for each $t\in \mathbb{R}$, the mapping $S_{t}=\exp(it L(a,a)) : E \to E$ is a surjective linear isometry (triple isomorphism) with inverse $S_{t}^{-1}=S_{-t}.$ Since $a\in F$ and $F$ is a JB$^*$-subtriple of $E$, we deduce that $L(a,a)^{n} (F) \subseteq F$, and hence $S_{t} (F) \subseteq F,$ for every $t\in \mathbb{R}$.\smallskip

Applying Lemma \ref{l surjective linear isometries} it follows that $S_t (\pi_{_F} (x)) = \pi_{_F} ( S_t (x)),$ for every $t\in \mathbb{R}.$ Having in mind that $a\perp x$ it follows that $L(a,a)^n (x) =0,$ for every natural $n$, which shows that $S_t (x) = x$ for every real $t$. Therefore $$\pi_{_F} (x ) =\pi_{_F} (x) + it L(a,a) (\pi_{_F} (x)) + \sum_{n=2}^{\infty} \frac{i^n t^n}{n!} L(a,a) (\pi_{_F} (x)).$$ Differentiating at $t=0$ we conclude that $L(a,a) (\pi_{_F} (x))=0$, or equivalently $a\perp  \pi_{_F} (x)$ (cf. \cite[Lemma 1]{BurFerGarMarPe08}).\smallskip

We have proved that $a\perp x, \pi_{_F} (x).$ Therefore $\pi_{_F} (x)+\mu a\in F$, for every $\mu\in \mathbb{C}$ and, by orthogonality, $$0<\hbox{dist} (x,F)= \|x-\pi_{_F} (x)\| = \max\{ \|x-\pi_{_F} (x)\| ,\|\mu a\|\} = \|x-\pi_{_F} (x) -\mu a\|,$$ for every $\mu \in \mathbb{C}$ with $\|\mu a\| \leq \|x-\pi_{_F} (x)\|,$ contradicting the uniqueness of the best approximation of $x$ in $F$.
\end{proof}

We recall that a tripotent $u$ in the bidual of a JB$^*$-triple $E$ is said to be \emph{open} when $E^{**}_2 (u)\cap E$ is weak$^*$ dense in $E^{**}_2 (u)$ (see \cite{EdRu96}). A tripotent $e$ in $E^{**}$ is said to be \emph{compact-$G_{\delta}$} (relative to $E$) if there exists a norm one element $a$ in $E$ such that $e$ coincides with $s(a)$, the support tripotent of $a$  (see \cite{EdRu96}). A tripotent $e$ in $E^{**}$ is said to be \emph{compact} (relative to $E$) if there exists a decreasing net $(e_{\lambda})$ of tripotents in $E^{**}$ which are compact-$G_{\delta}$ with infimum $e$, or if $e$ is zero.\smallskip

Closed and bounded tripotents in $E^{**}$ were introduced and studied in \cite{FerPe06} and \cite{FerPe07}.  A tripotent $e$ in $E^{**}$ such that $E^{**}_0 (e)\cap E$ is weak$^*$ dense in $E^{**}_0(e)$ is called \emph{closed} relative to $E$. When there exists a norm one element $a$ in $E$ such that $a= e + P_0 (e) (a)$, the tripotent $e$ is called \emph{bounded} (relative to $E$) (cf. \cite{FerPe06}).  Theorem 2.6 in \cite{FerPe06} (see also \cite[Theorem 3.2]{FerPe10}) asserts that a tripotent $e$ in $E^{**}$ is compact if, and only if, $e$ is closed and bounded.\smallskip

\begin{corollary}\label{c Peirce zero of a compact tripotent} Let $F$ be a \v{C}eby\v{s}\"{e}v JB$^*$-subtriple of a JB$^*$-triple $E$. Let $e$ be a tripotent in $F^{**}$ satisfying that $F^{**}_2 (e) \cap F\neq \{0\}$. Then $\{e\}_{E}^{\perp} = \{x\in E : x\perp e\} = E\cap E^{**}_0 (e) \subseteq F$. Furthermore, if $e$ is closed in $E^{**}$ we also have $E^{**}_0 (e) = F^{**}_0(e)$.
\end{corollary}

\begin{proof} By hypothesis, the set $F\cap F^{**}_2 (e)$ is non-zero, thus, there exists a non-zero element $a\in F\cap F^{**}_2 (e)$. It is easy to check that $\{e\}_{E}^{\perp} \subseteq \{a\}_{E}^{\perp}$, and the latter is contained in $F$ by Proposition \ref{p orthogonal}.\smallskip

We have already proved that $\{e\}_{E}^{\perp} = E\cap E^{**}_0 (e) \subseteq F$, which implies that $ E\cap E^{**}_0 (e) = F\cap F^{**}_0 (e)$. Since $e$ is closed in $E^{**}$, we can assure that $$E^{**}_0 (e) = \overline{E\cap E^{**}_0 (e)}^{\sigma(E^{**},E^*)}=  \overline{F\cap F^{**}_0 (e)}^{\sigma(E^{**},E^*)} \subseteq F^{**}_0 (e) \subseteq E^{**}_0 (e).$$
\end{proof}

\begin{corollary}\label{c Peirce zero of a range tripotent} Let $F$ be a \v{C}eby\v{s}\"{e}v JB$^*$-subtriple of a JB$^*$-triple $E$. Let $a$ be a non-zero element in $F$ and let $r(a)$ denote the range tripotent of $a$ in $F^{**}$. Suppose that $\{a\}^{\perp}_{F}\neq \{0\}$. Then $E^{**}_0 (r(a)) = F^{**}_0(r(a))$.
\end{corollary}

\begin{proof} We can assume that $\|a\|=1$. Let $F_a$ denote the JB$^*$-subtriple of $F$ (or of $E$) generated by $a$. We have already commented that there exists $L_a\subseteq [0,1]$, with $L_a$ compact and $1\in L_a$, and a triple isomorphism $\varPsi$ from $F_a$ onto $C_{0}(L_a)$, satisfying $\varPsi(a) (t) = t$ $(t\in L_a)$  (cf. \cite[Lemma~1.14]{Ka83}).\smallskip

Proposition \ref{prop inner ideals} and Corollary \ref{c prop inner ideals} imply that $E(a) = F(a).$ Therefore $E(a)=F(a)$ is an open JB$^*$-subtriple of $E^{**}$ relative to $E$ in the sense employed in \cite{FerPe06,FerPe09,FerPe10}. Proposition 3.3 in \cite{FerPe10} (or \cite[Corollary 2.9]{FerPe06}) implies that every compact tripotent in $F(a)^{**}$ is compact in $E^{**}.$ Let us take a compact tripotent $e\in F(a)^{**}$ satisfying that $e\leq r(a)$ and $F^{**}_2 (e) \cap F\neq \{0\}$ (we can consider, for example $e= \chi_{_{[\delta,1]\cap L_a}}$ the characteristic function of the set $[\delta,1]\cap L_a$ in $F_a$ and $y (t):=\left\{%
\begin{array}{ll}
    0, & \hbox{if $t\in [0,\delta ]$;} \\
    \hbox{affine}, & \hbox{if $t\in [\delta,2 \delta ]$;} \\
    1, & \hbox{if $t\in [2 \delta,1 ]$.} \\
\end{array}%
\right.$ in $F_a$, $y\in F^{**}_2 (e) \cap F$ with $0<\delta<2\delta<1).$ Since $e$ is compact, and hence closed in $E^{**}$ (cf. \cite[Theorem 2.6]{FerPe06}), Corollary \ref{c Peirce zero of a compact tripotent} proves that $E^{**}_0 (e) = F^{**}_0(e).$ Finally, it is easy to see that, since $r(a)\geq e$, $E^{**}_0 (r(a))\subseteq E^{**}_0 (e) = F^{**}_0(e)\subseteq F^{**},$ and hence $E^{**}_0 (r(a)) = F^{**}_0(r(a))$.
\end{proof}

We turn now our focus to the Peirce one subspace associated with a range tripotent.

\begin{lemma}\label{l Peirce 1} Let $F$ be a \v{C}eby\v{s}\"{e}v JB$^*$-subtriple of a JB$^*$-triple $E$. Suppose $a_1,a_2, b_1,b_2$ are elements in $F$ satisfying that  $a_1+b_1$ and $a_2+b_2$ are not Brown-Pedersen quasi-invertible in $F$, and $a_i \perp b_j$ for every $i,j\in\{1,2\}$. Then $ L(b_1,b_2) L(a_1,a_2) (E) \subseteq F.$
\end{lemma}

\begin{proof} Since $a_1+b_1$ and $a_2+b_2$ are not Brown-Pedersen quasi-invertible in $F$, Proposition \ref{prop inner ideals} assures that \begin{equation}\label{eq 20052015} Q(a_j,b_j) (E) \subseteq Q(a_j+b_j) (E) + Q(a_j) (E) + Q(b_j)(E) \subseteq F,
 \end{equation}for every $j=1,2.$ The identity $$Q(a_1,b_1) Q(a_2, b_2) = L(b_1,b_2) L(a_1,a_2)$$ can be easily deduced from the Jordan identity and the orthogonality of $a_i$ and $b_j$. The last identity together with \eqref{eq 20052015} prove the desired statement.
 \end{proof}

We can establish now our first main result on \v{C}eby\v{s}\"{e}v JB$^*$-subtriples of a general JB$^*$-triple.

\begin{thm}\label{t F rank 3 Cebysev subtriple} Let $F$ be a \v{C}eby\v{s}\"{e}v JB$^*$-subtriple of a JB$^*$-triple $E$. Suppose $F$ has rank greater or equal than three. Then $E=F$.
\end{thm}

\begin{proof} Since $F$ has rank greater or equal than three, we can find mutually orthogonal norm-one elements $a,b,c$ in $F.$ Let $F_a$, $F_b$, and $F_c$ denote the JB$^*$-subtriples of $F$ generated by $a$, $b$, and $c,$ respectively. Since $F$ is a JB$^*$-subtriple of $E$, they also coincide with the JB$^*$-subtriples of $E$ generated by $a$, $b$, and $c$ respectively. We observe that $c$ is orthogonal to every element in $F_a\oplus F_b$. So, given $a_1,a_2\in F_a$ and $b_1,b_2\in F_b$, we apply Lemma \ref{l Peirce 1} to deduce that \begin{equation}\label{eq 1 21052015} L(a_1,a_2) L(b_1,b_2) (E) \subseteq F,
\end{equation} for every $a_1,a_2\in F_a$ and $b_1,b_2\in F_b$. \smallskip

Let us consider two bounded sequences $(a_n)\subset F_a$, $(b_n)\subset F_b$ converging to $r(a)$ and $r(b)$ (the range tripotents of $a$ and $b$ in $F^{**}$) in the strong$^*$-topology of $F^{**}$, or equivalently, in the strong$^*$-topology of $E^{**}$ (cf. \cite[Corollary]{Bun01}). Fix $x\in E$. It follows from \eqref{eq 1 21052015} that $L(a_n,a_n) L(b_n,b_n) (x) \in F$, for every natural $n$. The joint strong$^*$-continuity of the triple product on bounded sets of $E^{**}$ implies that $L(r(a),r(a)) L(r(b),r(b)) (x) \in F^{**}\equiv \overline{F}^{\sigma(E^{**},E^{*})},$ for every $x\in E.$ The weak$^*$-density of $E$ in $E^{**}$ (cf. Goldstine's theorem) and the separate weak$^*$-continuity of the triple product of $E^{**}$ assures that \begin{equation}\label{eq 2 21052015} L(r(a),r(a)) L(r(b),r(b)) (E^{**}) \subseteq F^{**}.
\end{equation}

It is well known that $L(r(a),r(a)) = P_2 (r(a)) + \frac12 P_1 (r(a))$, and a similar identity holds for $r(b)$. Therefore \eqref{eq 2 21052015} implies that \begin{equation}\label{eq 3 21052015} \left(P_2 (r(a)) + \frac12 P_1 (r(a))\right) \left(P_2 (r(b)) + \frac12 P_1 (r(b))\right) (E^{**} ) \subseteq F^{**}.
 \end{equation}It is well known that $r(a)\perp r(b)$ implies $$E^{**}_2 (r(a)+r(b)) = E^{**}_2 (r(a))\oplus E^{**}_2 (r(b)) \oplus \Big(E^{**}_1 (r(a))\cap E^{**}_1 (r(b))\Big)$$ and $$E^{**}_1 (r(a)+r(b)) = \Big(E^{**}_1 (r(a)) \cap E^{**}_0 (r(b))\Big) \oplus \Big(E^{**}_0 (r(a))\cap E^{**}_1 (r(b))\Big).$$ We thus deduce from \eqref{eq 3 21052015} that $$ E^{**}_2 (r(a)+r(b)) \subseteq F^{**},\hbox{ and } E^{**}_1 (r(a)+r(b))\subseteq F^{**},$$ and hence \begin{equation}\label{eq Peirce 12} E^{**}_2 (r(a)+r(b)) = F^{**}_2 (r(a)+r(b)),
 \end{equation} and $$ E^{**}_1 (r(a)+r(b))= F^{**}_1 (r(a)+r(b)).$$

Since $c\in \{a+b\}_{F}^{\perp}$, we are in position to apply Corollary \ref{c Peirce zero of a range tripotent} to show that \begin{equation}\label{eq Peirce 0} E^{**}_0 (r(a)+r(b))= E^{**}_0 (r(a+b)) = F^{**}_0(r(a+b)) =  F^{**}_0(r(a)+r(b)).
\end{equation} It follows from \eqref{eq Peirce 12} and \eqref{eq Peirce 0} that $$E^{**} = E^{**}_0 (r(a+b))\oplus E^{**}_1 (r(a+b)) \oplus E^{**}_2 (r(a+b))$$ $$ = F^{**}_0 (r(a+b))\oplus F^{**}_1 (r(a+b)) \oplus F^{**}_2 (r(a+b)) =F^{**}.$$  Finally, it is easy to check that, under these conditions, $E=F,$ as desired.
\end{proof}

\begin{corollary}\label{c t F rank 3 Cebysev subtriple} Let $B$ be a \v{C}eby\v{s}\"{e}v C$^*$-subalgebra (respectively, a \v{C}eby\v{s}\"{e}v JB$^*$-subtriple) of a C$^*$-algebra $A$. Suppose $B$ has rank greater or equal than three. Then $A=B$.$\hfill\Box$
\end{corollary}

It is well known that every infinite dimensional C$^*$-algebra contains an infinite sequence of mutually orthogonal non-zero elements (cf. \cite[Exercise 4.6.13]{KadRing1997}), that is every infinite dimensional C$^*$-algebra has infinite rank. Furthermore, Exercise 4.6.12 in \cite{KadRing1997} proves that every C$^*$-algebra with finite rank must be finite dimensional and hence unital (cf. \cite[Theorem I.11.2]{Takesaki}).\smallskip

Let us observe that a non-zero C$^*$-algebra without unit must have infinite rank. Thus, the following result of Pedersen follows from Corollary \ref{c t F rank 3 Cebysev subtriple}.

\begin{corollary}\label{c Pedersen}\cite[Theorem 4]{Ped79} Let $B$ be a non-unital \v{C}eby\v{s}\"{e}v C$^*$-subalgebra of a C$^*$-algebra $A$. Then $A=B$.$\hfill\Box$
\end{corollary}

It remains to study \v{C}eby\v{s}\"{e}v JB$^*$-subtriples of rank smaller or equal than two. In this case, the conclusion will follow from the main result in \cite{JamPeSiTah2014Ceby} and the studies about finite rank JB$^*$-triples developed in \cite{BuChu} and \cite{BeLoPeRo}.\smallskip

\begin{thm}\label{t Cebysev JB*-subtriples JPST} Let $F$ be a non-zero \v{C}eby\v{s}\"{e}v JB$^*$-subtriple of a JB$^*$-triple $E$. Then exactly one of the following statements holds:\begin{enumerate}[$(a)$]\item $F$ is a rank one JBW$^*$-triple with dim$(F)\geq 2$ (i.e. a complex Hilbert space regarded as a type 1 Cartan factor). Moreover, $F$ may be a closed subspace of arbitrary dimension and $E$ may have arbitrary rank;
\item $F= \mathbb{C} e$, where $e$ is a complete tripotent in $E$;
\item $E$ and $F$ are rank two JBW$^*$-triples, but $F$ may have arbitrary dimension;
\item $F$ has rank greater or equal than three and $E=F$.
\end{enumerate}
\end{thm}

\begin{proof}  If $F$ has rank $\geq 3$, Theorem \ref{t F rank 3 Cebysev subtriple} implies that $E=F$. We may therefore assume that $F$ has rank $\leq 2$. It follows from \cite[Proposition 4.5 and comments at the beggining of \S 4]{BuChu} (see also \cite[\S 3]{BeLoPeRo}) that $F$ is reflexive. So, $F$ is a reflexive JBW$^*$-triple of rank $\leq 2$.\smallskip

We shall adapt next the arguments in the proof of \cite[Theorem 13]{JamPeSiTah2014Ceby}, we provide a simplified argument. Every JBW$^*$-triple admits an abundant collection of complete tripotents or extreme points of its closed unit ball (cf. \cite[Lemma 4.1]{BraKaUp78} and \cite[Proposition 3.5]{KaUp77} or \cite[Theorem 3.2.3]{Chu}). Thus, we can find a complete tripotent $e$ in $F$. There are only two possibilities:  either $e$ is minimal in $F$ or $e$ has rank two in $F$.\smallskip

When $e$ is rank two in $F$. We can write $e= e_1+e_2$ with $e_1,e_2$ mutually orthogonal minimal tripotents in $F$. Proposition \ref{prop JamPeSiddTahCeby} proves that $E_2 (e_j) = F_2 (e_j) = \mathbb{C} e_j$, $E_0 (e_j) = F_0 (e_j)$, and $E_0 (e_1+e_2) = F_0 (e_1+e_2)=\{0\},$ which proves that $e_1$ and $e_2$ are minimal tripotents in $E$, $e$ is complete in $E$, and $E$ is a rank-2 JBW$^*$-triple.\smallskip

We finally assume that $e$ is minimal and complete in $F$. If dim$(F)=1$, then $F= \mathbb{C} e$, and we are in case $(b)$, otherwise we are in case $(a)$.
\end{proof}

It should be remarked here that Remark 7 in \cite{JamPeSiTah2014Ceby} provides an example of an infinite dimensional rank-one \v{C}eby\v{s}\"{e}v JB$^*$-subtriple of a JB$^*$-triple, while \cite[Remark 13]{JamPeSiTah2014Ceby} gives an example of a rank-one \v{C}eby\v{s}\"{e}v JB$^*$-subtriple of a rank-$n$ JBW$^*$-triple, where $n$ is an arbitrary natural number.\smallskip

In the setting of C$^*$-algebras, the following result of Pedersen follows directly from our Theorem \ref{t Cebysev JB*-subtriples JPST}.

\begin{corollary}\cite[Theorem 5]{Ped79} If $A$ is a unital C$^*$-algebra and $B$ is a \v{C}eby\v{s}\"{e}v C$^*$-subalgebra of $A$, then $A=B,$ $B= \mathbb{C} 1$ or $A=M_2 (\mathbb{C})$ and $B$ is the subalgebra of diagonal matrices.$\hfill\Box$
\end{corollary}

\end{document}